\documentclass{article}
\usepackage[utf8]{inputenc}
\usepackage[margin=1in]{geometry}
\usepackage{amsmath} % assumes amsmath package installed
\usepackage{amsthm}
\usepackage{todonotes}
\usepackage[nolist]{acronym}
\usepackage{xcolor}
\usepackage{amsmath}
\usepackage{amssymb}
\usepackage{mathtools}
\usepackage{subcaption}
\usepackage{hyperref}
\hypersetup{colorlinks=true, linkcolor=black}
\usepackage[ruled]{algorithm2e}
\usepackage{mathrsfs}
\definecolor{ccolor}{RGB}{203,96,21}

\newcommand{\R}{\mathbb{R}}

\newcommand{\dc}{\mathcal{D}}

\DeclarePairedDelimiter{\norm}{\lVert}{\rVert}

\DeclarePairedDelimiterX{\inp}[2]{\langle}{\rangle}{#1, #2}

\newcommand{\bm}[1]{{\mathbf{#1}}}
\usepackage{algorithmicx}
\newtheorem{theorem}{Theorem}
\newtheorem{lemma}{Lemma}
\newtheorem{problem}{Problem}
\newtheorem{remark}{Remark}
\newtheorem{example}{Example}

\newtheorem{definition}{Definition}
\newcommand{\beq}{\begin{equation}}
\newcommand{\eeq}{\end{equation}}

\newcommand{\vt}[1]{{\bm{#1}}}
\newcommand{\mt}[1]{{\bm{#1}}}

\renewcommand\footnotemark{}

\title{\LARGE \bf Robust Data-Driven Safe Control \\ using Density Functions}

\author{Jian Zheng$^1$,  Tianyu Dai$^2$, Jared Miller$^1$, Mario Sznaier$^1$
\thanks{$^1$J. Zheng, J. Miller, and M. Sznaier are with the Robust Systems Lab,  ECE Department, Northeastern University, Boston, MA 02115. (e-mails: \{zheng.jian1, miller.jare\}@northeastern.edu, msznaier@coe.neu.edu)}
\thanks{$^2$T. Dai is with The MathWorks, Inc., 1 Apple Hill Drive,
		Natick, MA 01760 USA (e-mail: tdai@mathworks.com)}
\thanks{J. Zheng, J. Miller, and M. Sznaier were partially supported by NSF grants  CNS--1646121, ECCS--1808381 and CNS--2038493, AFOSR grant FA9550-19-1-0005, and ONR grant N00014-21-1-2431.  
}}

\begin{document}

\maketitle

\begin{abstract} \label{sec:abstract}
This paper presents a tractable framework for data-driven synthesis of robustly safe control laws. Given noisy experimental data and some priors about the structure of the system, the goal is to synthesize a state feedback law such that the trajectories of the closed loop system are guaranteed to avoid an unsafe set even in the presence of unknown but bounded disturbances (process noise).  The main result of the paper shows that for polynomial dynamics, this problem can be reduced to a tractable convex optimization by combining elements from polynomial optimization and the theorem of alternatives. This optimization provides both a rational control law and a density function safety certificate. These results are illustrated with numerical examples. 
\end{abstract}
\section{Introduction} \label{sec:introduction}

The goal of this paper is to develop a tractable framework for data-driven synthesis of safe control laws that are robust to $L_\infty$-bounded noise  in both data-collection and during execution. Specifically,  given noisy experimental data generated by an unknown system and some priors about its structure, the objective is to synthesize a state feedback law such that the trajectories of the closed loop system starting in a given initial condition set $\mathcal{X}_0$ are guaranteed to avoid an unsafe set $\mathcal{X}_u$, even in the presence of unknown but bounded disturbances. 
Our main result shows that, for polynomial dynamics, the safe \ac{DDC} problem can be posed as the feasibility of a \ac{SOS} program. A substantial reduction in the number of variables involved (and hence computational complexity) is achieved by exploiting the theorem of alternatives, leading to a \ac{SDP} that provides both a density-function based control law and a robust safety certificate. 

Safety verification and synthesis of safe control laws have been the subject of intense research during the past decade. Level-set methods separate the initial and unsafe set by the $0$-contour of a solved function.
Barrier functions \cite{prajna2004safety} are a level-set method to certify the safety of trajectories, given that the superlevel sets of the barrier function are invariant. This superlevel invariance can be relaxed through slack (class-$\mathcal{K}$) conditions, while ensuring that the $0$-level set is invariant \cite{ames2019control, xiao2019control}. The level-set certificate of stability may be solved jointly with a safety-guaranteeing control policy $u(\cdot)$ (\ac{CBF}). When a barrier function is given, the min-norm controller will ensure safety of trajectories, and can be found through quadratic programming \cite{ames2014control}. Robustness of given barrier functions to disturbances may be analyzed using input-to-state stability \cite{XU2015robustcbf}.
Barrier functions and funnels \cite{majumdar2013control} contain bilinearities when jointly synthesizing controllers and barriers.
An alternative level-set certificate is  Density \cite{rantzer2004analysis} functions, which are based on Dual Lyapunov methods for stability \cite{rantzer2001dual}. Controllers and density functions can be simultaneously solved in a convex manner. In some systems, density functions may exist and provide improved performance as compared to  barrier functions  \cite{chen2020densityvalue}.

We briefly compare against other methods of safety-constrained control.
Interval analyses, such as Mixed Monotonicity \cite{coogan2020mixed}, offer real-time performance at the expense of conservatism in safe generation.
Hamilton-Jacobi reachability \cite{bansal2017hamilton} performs forward and backward reachable set analysis based on level sets of a differential games' value function, whose computation could require solving PDEs or neural net approximations.
Reinforcement Learning necessitates training and prior information of safety properties (e.g. Lipschitz bounds on dynamics), and does not generally exploit physical principles and model structure \cite{brunke2022safe}. Koopman methods leverage the predictive capabilities of nonlinear models, 
%\cite{korda2020optimal, otto2021koopman}, 
but they contain error bounds that can conflict against safety certification \cite{folkestad2020data}.

% e\textcolor{red}{need a brief description of these and why are we not using them}

% \urg{Literature review here.}

% \urg{Safe Control}

% \urg{Data-Driven Control, Data-Driven Safety}

\ac{DDC} is a methodology that synthesizes control laws directly from acquired system observations and skips a system-identification/robust-synthesis pipeline \cite{formentin2014comparison}. Amongst the vast literature in \ac{DDC}, the closest approaches related to the present paper are those that pursue a set membership approach, which seeks to find a controller that stabilizes the set of all plants compatible with the observed data (the consistency set) \cite{dai2018moments,dai2020semi,waarde2020noisy,martin2021data,berberich2021robustmpc, bisoffi2022data, miller2022lpvqmi, miller2022eiv_short}. These approaches provide a controller together with a stability certificate, usually in the form of a common Lyapunov function.  Further, the methods can be extended to provide worst case performance bounds (e.g. the $H_2,H_\infty$ or $L_\infty$ sense), over the set of data-consistent plants.  However, these approaches cannot handle safety constraints beyond those expressed in terms of these norms.

Recent work on \ac{DDC} under safety constraints includes \cite{rosolia2018learning, lopez2021robust, dacs2022robust}. 
The method in \cite{rosolia2018learning} performs iterative model predictive control for a discrete-time system by constraining state trajectories to always lie in a sampled safe set (using integer programming). 
The work in \cite{lopez2021robust} uses contraction methods to form robust adaptive \acp{CBF} under a set membership approach, but assumes that the input relation $g(\cdot)$ is known. The approach in \cite{dacs2022robust} uses a disturbance observer to provide robust \acp{CBF} by separating known and unknown dynamics. In our setting, we assume only prior knowledge of the system model (polynomial up to a specified degree) and cannot generally provide this separation. 
%\textcolor{red}{What are the shortcomings of these approaches? why is ours needed?} 
Our work involves continuous-time dynamics and interpretable (density) certificates of robust safety.
To the best of our knowledge, our approach is the first \ac{DDC} method under safety constraints that simultaneously considers data-collection and online-dynamics noise.
% \vspace{-2cm}
% \urg{(outside of $H_2$ noise?)}.

% \urg{What makes our algorithm novel.}

Contributions of this work are,
\begin{itemize}
    \item A \ac{DDC} framework for density-based robust safe control.
    \item Tractable synthesis of robustly safe density functions by exploiting the theorem of alternatives.
    %using an Extended Farkas Lemma 
    \item Numerical examples demonstrating robustly safe control on polynomial systems.
\end{itemize}

This paper has the following structure: 
Section \ref{sec:preliminaries} reviews preliminaries such as notation, density functions for safety, 
%the Extended Farkas Lemma for polytope containtment
and \ac{SOS} polynomials.
% \urg{SOS polynomials could also be defined in Section \ref{sec:sos_safety}}. 
Section \ref{sec:safe_no_w} performs data-driven synthesis of safe controllers using density functions and \ac{SOS} methods in the case where $L_\infty$-bounded noise  occurs at data collection and  the dynamics are subject to unknown but bounded disturbances. 
% Section \ref{sec:safe_with_w} extends the \ac{SOS}-density method towards the synthesis of robustly safe controllers where $L_\infty$-bounded noise also affects system execution (process noise). 
Section \ref{sec:examples} demonstrates the effectiveness of our approach  on several  example systems. Section \ref{sec:conclusion} concludes the paper.

% \urg{Fill in the paper structure}
% Section \ref{sec:preliminaries} will review preliminaries such as notation, notions of stability for linear systems, and \ac{SOS} proofs of polynomial nonnegativity. Section \ref{sec:full_method} will present 
% The paper is concluded in Section \ref{sec:conclusion}.
\section{Preliminaries} \label{sec:preliminaries}
\begin{acronym}
% \acro{BSA}{Basic Semialgebraic}
\acro{CBF}{Control Barrier Function}

\acro{DDC}{Data Driven Control}

% \acro{GAS}{Globally Asymptotically Stable}

% \acro{CSP}{Correlative Sparsity Pattern}

\acro{LMI}{Linear Matrix Inequality}
\acroplural{LMI}[LMIs]{Linear Matrix Inequalities}
\acroindefinite{LMI}{an}{a}
% \acp{LMI}: LMIs \iacf

\acro{LP}{Linear Program}
\acroindefinite{LP}{an}{a}
% \acro{OCP}{Optimal Control Problem}

% \acro{ODE}{Ordinary Differential Equation}

% \acro{POP}{Polynomial Optimization Problem}

% \acro{PSD}{Positive Semidefinite}

% \acro{PD}{Positive Definite}

% \acro{PDE}{Partial Differential Equation}

\acro{SDP}{Semidefinite Program}
\acroindefinite{SDP}{an}{a}

\acro{SOS}{Sum of Squares}
\acroindefinite{SOS}{an}{a}
% \iac{} an SOS \Iac An SOS
% \acp

% \acro{WSOS}{Weighted Sum of Squares}

\end{acronym}

\subsection{Notation} \label{subsec:notation}

\begin{tabular}{p{0.13\columnwidth}p{0.75 \columnwidth}}
$\R^n$ & Set of $n$-tuples of real numbers\\
$x,\vt{x},\mt{X}$ & Scalar, vector, matrix\\
$\mathbf{1,0},\mt{I}$ & Vector/matrix of all 1s, 0s, identity matrix\\
$\left\|\vt{x} \right\|_\infty$ & $L_\infty$-norm of vector $\vt{x}$\\
%$\left\|\mt{X} \right\|_p$ & Induced $\ell_p$-norm of matrix $\mt{X}$\\
$\mt{X}\succeq 0$ & $\mt{X}$ is positive semi-definite \\
$\otimes$ & Kronecker product\\
\text{vec}$(\mt{X})$ & Vectorized matrix along columns: \\ & $\text{vec}(\mt{X})= \left[ \mt{X}(\colon,1)^T,\ldots,\mt{X}(\colon,n)^T\right]^T$ \\
$\rho \in C^d$ & $\rho$ has a continuous $d^{th}$ derivative\\
$\nabla \rho$ & Gradient of scalar function $\rho$ \\
$\nabla \cdot f$ & Divergence of vector function $f$\\
% $\Sigma[x]$ & Cone of \ac{SOS} polynomials in $x \in \R^n$\\
% $\Sigma[x]_{\leq d}$ &  \ac{SOS} polynomials up to degree $d$\\
% \R[x]
\end{tabular}

\subsection{Sum-of-Squares}

We briefly review the concept of \ac{SOS} polynomials and proofs of nonnegativity \cite{parrilo2000structured}. A polynomial $p \in \R[\vt{x}]$ is \ac{SOS} (and hence nonnegative) if there exist polynomials $\{q_\ell \in \R[\vt{x}]\}_{\ell= 1}^L$ such that $p(\vt{x}) = \sum_{\ell=1}^L q_\ell(\vt{x})^2$.
%Given that each $q_\ell$ has real coefficients and $q_\ell(x)^2$ is nonnegative, it holds that the sum $\sum_{\ell=1}^L q_\ell(x)^2$ is also nonnegative. 

The cone of \ac{SOS} polynomials is $\Sigma[\vt{x}]$, and its up to degree $2d$ restriction is $\Sigma_{d}[\vt{x}]$. The cone $\Sigma_d[\vt{x}]$ is semidefinite representable as $p(\vt{x}) = \vt{v}(\vt{x})^T \mt{Q} \vt{v}(\vt{x})$ where $\vt{v}(\vt{x})$ is the monomial vector up to degree $d$ and $\mt{Q} \succeq 0$ is the Gram matrix. A sufficient condition for a polynomial $p$ to be nonnegative over the semialgebraic region $\{\vt{x} \mid h_i(\vt{x}) \geq 0, \ i = 1..N_c\}$ is that $p$ is contained in the quadratic module formed by $h_i$ (there exists $\sigma_{0..N_c} \in \Sigma[\vt{x}]$ such that $p(\vt{x}) = \sigma_0 + \sum_{i=1}^{N_c}\sigma_i h_i$).

% \begin{itemize}
%     \item[A1] Assumption 1
%     \item[A2]  Assumption 2
% \end{itemize}
% \subsection{SOS Polynomials}
% This paper shows that the data-driven safe control synthesis can be formulated as a polynomial nonnegativity  problem, which is generally NP-hard (for polynomials of even degree $d\geq 4$). This difficulty can be handled by relaxing the problem to establishing that a given polynomial is a Sum of Squares. Consider a polynomial $f$ of degree $2m$. A sufficient condition to $\forall x \in \R^n: f(x)\geq 0$ is that $f$ can be written as a \ac{SOS} 
% $f = \textstyle \sum_i g_i^2.$
% Let $v$ be a vector with all monomials of degree less than or equal to $m$. Then $f$ is \ac{SOS} if and only if there exists a positive semidefinite matrix $Q$ such that
% $f = v^T Qv$ \cite{parrilo2000structured}.
% Comparing term on the both sides leads to an \ac{SDP} program, which can be solved by off-the-shelf solvers. \textcolor{red}{do we need to mention the Psatz?}

\subsection{Level-Set-Based Safety Certification}
Consider a continuous-time  system of the form
\begin{equation} \label{eq:dynamics}
    \dot{\vt{x}} = f(\vt{x},\vt{w})
\end{equation}
where $\vt{x}\in \R^n$ is the state and $\vt{w} \in \mathcal{W} \subseteq \R^n$ is a disturbance. Further, assume that $\vt{w}(t)$ is such that the trajectories of \eqref{eq:dynamics} are well defined for any initial condition $\vt{x}_0 \in \mathcal{X}_0$. In the sequel, we will denote these trajectories as $\vt{x}(t,\vt{w},\vt{x}_0)$.
\begin{definition}
Given an initial condition set $\mathcal{X}_0 \subseteq \R^n$ and an unsafe set $\mathcal{X}_u \subseteq \R^n$, system \eqref{eq:dynamics} is robustly safe if, for all $t$, all initial conditions $\vt{x}_0 \in \mathcal{X}_0$ and all $\vt{w}(t) \in \mathcal{W}$,  $\vt{x}(t,\vt{w},\vt{x}_0) \not \in\mathcal{X}_u$.
 \end{definition}
Typically, safety is certified through the use of barrier functions, defined as:
\begin{definition}
A differentiable function $B(\vt{x}): \R^n \to \R$ is a robust barrier function for \eqref{eq:dynamics} with respect to $\mathcal{X}_0$ and $\mathcal{X}_u$ if 
\begin{align}
B(\vt{x}) & \leq 0,\; \forall \vt{x} \in \mathcal{X}_0, \; 
B(\vt{x})  >0, \; \forall \vt{x} \in\mathcal{X}_u \label{eq:B2} \\
\frac{\partial B}{\partial \vt{x}}f(\vt{x},\vt{w}) & < 0, \; \forall \vt{w} \in \mathcal{W} \quad \text{whenever $B(\vt{x})=0$.} \label{eq:B3}
\end{align}
\end{definition} 
As shown for instance in \cite{prajna2004safety}, existence of a barrier function is a sufficient condition to certify safety. Note however that the conditions above are non-convex, even when $\vt{w} \equiv  0$, due to the constraint \eqref{eq:B3}. For instance, in the case of polynomial dynamics and semialgebraic $\mathcal{X}_0$ and $\mathcal{X}_u$, if $B(\vt{x})$ is also polynomial,  this constraint can be enforced by introducing a polynomial multiplier $h(\vt{x})$ and imposing that
\begin{equation}\label{eq:CBF}
-\frac{\partial B}{\partial \vt{x}}f(\vt{x},\vt{w}) + h(\vt{x})B(\vt{x}) \in \Sigma[\vt{x}]. \end{equation}
The condition above cannot be written as a single semi-definite optimization due to the multiplication of the coefficients of the two unknown polynomials, $h$ and $B$. Possible relaxations include choosing a fixed multiplier $h$, or simply dropping the $B(\vt{x})=0$ quantifier \cite{ames2019control}.
An alternative, convex approach based on the use of densities was proposed in \cite{rantzer2004analysis}.
\begin{theorem}[\cite{rantzer2004analysis}] \label{thm:safety}
Given $\mathcal{X}_0$ and $\mathcal{X}_u$, system \eqref{eq:dynamics} is safe if there exists a scalar function $\rho(\vt{x})\in C^1$ such that
\begin{subequations} \label{eq:safe_rho}
\begin{align}
\nabla \cdot [\rho(\vt{x}) f (\vt{x})]&> 0, \ \forall \vt{x}\in \R^n \label{eq:den} \\
\rho (\vt{x}) & \geq 0, \ \forall \vt{x}\in \mathcal{X}_0, \;
\rho (\vt{x}) < 0, \ \forall \vt{x}\in \mathcal{X}_u.
\end{align}
\end{subequations}
\end{theorem}
The advantage of this approach is that it leads to a convex problem in $\rho$. On the other hand, imposing that the divergence condition holds everywhere can be unnecessarily conservative.  

The concepts above can be easily extended to the case where the goal is to synthesize a control action that keeps a system safe by introducing the concept of \acp{CBF}.
%{ Control Barrier Functions}. 
\begin{definition}  A function $B(\vt{x})$ is a \ac{CBF} for the system
%\begin{equation} \label{eq:dynamicsu}
  $  \dot{\vt{x}} = f(\vt{x},u,\vt{w})$
%\end{equation}
if there exists a control law $u(\vt{x})$ such that $B(\vt{x})$ is a barrier function for the closed loop dynamics $\dot{\vt{x}}=f(\vt{x},u(\vt{x}),\vt{w})$.
\end{definition}
In principle, a \ac{CBF} and associated control law can be found by modifying \eqref{eq:CBF} to
\beq \label{eq:h_relaxation} -\frac{\partial B}{\partial \vt{x}}f(\vt{x},u(\vt{x}),\vt{w}) + h(\vt{x})B(\vt{x}) \in \Sigma[\vt{x}]. \eeq
Problem \eqref{eq:h_relaxation} is bilinear in the coefficients of $B, u$ even when restricted to polynomial dynamics and control laws and a fixed multiplier $h$, necessitating the use of relaxations.
% However, even in the case of polynomial dynamics and control laws and a fixed multiplier $h$, the problem above is bilinear in the coefficients of $u,B$, necessitating the use of relaxations. 
On the other hand, as shown in \cite{rantzer2004analysis}, the density based formulation can be easily modified to lead to  problems that are jointly convex in $\rho$ and $ \psi \doteq \rho u$.

\section{Data-Driven Safe Control} \label{sec:safe_no_w}

%%%-----------------------------
\subsection{Problem Statement}

% Consider the continuous-time control-affine system of the form
% \begin{equation} \label{eq:dynamics}
%     \dot{x} = f(x) + g(x) u
% \end{equation}
% where $x\in \R^n$ is the state, $u\in \R$ is the input, $f,g\colon \R^n \rightarrow \R^n$ are unknown polynomial functions, respectively. We assume the degrees of $f,g$ are known (or user-defined), then $f,g$ can be expressed by $f = F \phi$, $g = G \gamma$ where $F,G$ are unknown coefficient matrices and $\phi,\gamma$ are known monomial vectors in $x$, of proper dimensions. For example, for a two-state second-order polynomial system, $\phi,\gamma = [1,x_1,x_2,x_1^2,x_1x_2,x_2^2]^T$ and $F,G \in \R^{2\times 6}$.

The goal of this paper is to design a safe control law based on (noisy) experimental measurements for unknown polynomial systems where only minimal a-priori information  is available. Specifically, we consider control affine nonlinear  systems of the form
\begin{equation}\label{eq:dynamics_noise}
\dot{\vt{x}}(t) = f(\vt{x}) + g(\vt{x})u(t) + \vt{w}(t),
\end{equation}
where $u\in\R$ is the control and the input $\vt{w}$ satisfying $\forall t \geq 0 \colon \vt{w} \in \mathcal{W}$ represents an unknown random but bounded disturbance.  The only information available about the dynamics is that they can be expressed in terms of known dictionaries $\vt{\phi}(\vt{x}) \in \R^{d_f}, \vt{\gamma}(\vt{x}) \in \R^{d_g}$, that is
%\footnote{For simplicity, we consider single input, but our result can be extended to multiple inputs trivially.}
\beq\begin{aligned}\label{eq:pro1}
f(\vt{x}) = \mt{F} \vt{\phi}(\vt{x}); \; g(\vt{x})= \mt{G} \vt{\gamma}(\vt{x}) 
\end{aligned}\eeq
for some unknown  system parameter matrices $\mt{F} \in \mathbb{R}^{n\times d_f}$ and $\mt{G} \in \mathbb{R}^{n \times d_g}$.
In this context, the problem under consideration can be formally stated as:

\begin{problem} \label{pb:1} 
% Given $T$ measurements of the triple $(\dot{x}_s, x_s, u_s)$, $s=t_1,\ldots,t_T$, 
Given  a data-collection noise bound $\epsilon>0$, a process disturbance description $\vt{w} \in \mathcal{W}$ (e.g. $L_\infty$-bounded input), noisy derivative-state-input data $\dc = \{(\dot{\vt{x}}_s, \vt{x}_s, u_s)\}_{s=t_1..t_T}$ under the relation $\norm{\dot{\vt{x}}_s - f(\vt{x}_s) - g(\vt{x}_s)u_s}_\infty \leq \epsilon$, and basic semialgebraic sets  $\mathcal{X}_0$, $\mathcal{X}_u$, find a state-feedback control law $u(\vt{x})$  that renders all closed-loop  systems consistent with the observed data and priors   robustly safe with respect to $\mathcal{X}_0$ and $\mathcal{X}_u$, for all  $\vt{w} \in \mathcal{W}$.
\end{problem}

%\begin{remark}
   % The data $y_s$ is a noisy observation of the true derivative $\dot{x_s}$.
%\end{remark}

%\begin{remark}
%The solved control law $u(x)$ is not necessarily consistent with sample data $\dc$.
%\end{remark}

% In this section, we derive a data-driven algorithm to solve the problem \ref{pb:1}. 

%The data-driven Problem \ref{pb:1} will be approached by  define two polytopic sets in parameter space: one for $\dc$-consistent systems ($\mathcal{P}_1$) and one for safe systems ($\mathcal{P}_2$). Applying the Extended Farkas' Lemma to enforce set-containment leads to a polynomial feasibility problem, which can be solved with \ac{SOS} method by \ac{SDP} solvers.

\subsection{Model Based Safety} \label{sec:MBS}

In order to solve Problem \ref{pb:1}, in this section we first develop a  convex condition, less conservative than \eqref{eq:safe_rho}, that guarantees robust controlled safety of a model of the form \eqref{eq:dynamics_noise} assuming that $f(.)$ and $g(.)$ are known.

\begin{lemma}\label{lem:lemma2} Assume that the set $\mathcal{X}_u$ has a description of the form:
\[ \mathcal{X}_u \doteq \left \{\vt{x} \colon h_i(\vt{x}) \geq 0,\ i=1..N_c \right \}. \]
Then, if there exist scalar functions $\rho(\vt{x}), \psi(\vt{x}) \in C^1$ such that: (i) $u(\vt{x}) \doteq \frac{\psi(\vt{x})}{\rho(\vt{x})}$ is well defined over the safe region $\rho(\vt{x}) \geq 0$, (ii) for all $\vt{w} \in \mathcal{W}$ and initial condition $\vt{x}_0 \in \mathcal{X}_0$, the trajectories of \eqref{eq:dynamics_noise} are well defined, and (iii)
the following conditions hold:
\begin{subequations}\label{eq:MBS}
\begin{align}
&\nabla \cdot [\rho(\vt{x}) \left(f (\vt{x}) +\vt{w}\right ) + \psi(\vt{x}) g(\vt{x})] - \rho(\vt{x})h(\vt{x}) > 0 \label{eq:MBS1} \\
& \text{$\forall \vt{x}\in \R^n$ and $\vt{w} \in \mathcal{W}$} \nonumber \\
&\rho (\vt{x}) \geq 0, \ \forall \vt{x}\in \mathcal{X}_0, \;
\rho (\vt{x})< 0, \ \forall \vt{x}\in \mathcal{X}_u. 
%&u(x) \doteq \frac{\psi(x)}{\rho(x)} \; \text{bounded for all finite $x$}
\end{align}
\end{subequations}
where $h\doteq \min_i \left \{ h_i(\vt{x}) \right \}$,
then the control law $u(\vt{x})$ renders the closed loop system robustly safe with respect to $\mathcal{X}_u$.
\end{lemma}
\begin{proof} Since by assumption $\rho,\psi \in C^1$ and $u$ is well defined, \eqref{eq:MBS1} is equivalent to (omit $\vt{x}$):
\beq 
\frac{\partial \rho}{\partial \vt{x}}(f+gu+\vt{w}) + \rho \left (\nabla \cdot (f+gu) -h \right ) > 0
\eeq
where we used the fact that $\psi = \rho u$. Hence, for all $\vt{w} \in \mathcal{W}$,
\[ \frac{d\rho}{dt}+ \rho\left (\nabla \cdot (f+gu) -h \right ) > 0\]
along the closed loop trajectories, which implies that $ \frac{d\rho}{dt} > 0$ when $\rho[\vt{x}(t)]=0$.  Assume that there exists  a trajectory  $\vt{x}(t,\vt{x}_0,\vt{w})$ that starts at $\vt{x}_0 \in \mathcal{X}_0$ and such that $\vt{x}(T,\vt{x}_0,\vt{w}) \in \mathcal{X}_u$. By continuity, there exists some $0<t_1<T$ and some $dt$ such that $\rho(t_1)=0$ and $\rho(t)<0$ for all $t\in [t_1,t_1+dt]$. However, this contradicts the fact that  $\frac{d\rho}{dt} \left |_{t=t_1} > 0 \right.$.
\end{proof}
\begin{remark} Since $\min_i\left \{ h_i(\vt{x}) \right \}$ has a semialgebraic representation, finding polynomial functions $\rho$ and $\psi$ reduces to 
\iac{SOS} optimization via standard 
%Putinar Positivstellensatz 
arguments.
\end{remark}
\begin{remark}  Problem \eqref{eq:MBS} is an infinite-dimensional \ac{LP} in the values of $(\rho, \psi)$ at each $\vt{x}$, possessing both strict and non-strict inequality constraints. When compared against \eqref{eq:h_relaxation}, this formulation  has two advantages: (i) it avoids using an arbitrary, fixed multiplier $h(\vt{x})$, and (ii) it leads to jointly convex (in $\rho$ and $\psi$) optimization problems when searching for a control barrier and associated control action. On the other hand, \eqref{eq:MBS}, while retaining the desirable convexity properties of   \eqref{eq:safe_rho}, is less conservative: since the second term in \eqref{eq:MBS1} is nonnegative over the safe region, it does not require the first term to be positive everywhere, as is the case with  \eqref{eq:safe_rho}. Note that any feasible solution to  \eqref{eq:safe_rho} is also feasible for \eqref{eq:MBS}.
\end{remark}

\subsection{Safe Data Driven Control}\label{sec:SDDC}

This section presents the main result of the paper: a tractable, convex reformulation of Problem \ref{pb:1}.  In order to present these results, we begin by presenting  a tractable characterization of all systems that could have generated the observed data. 

Assume the sample data $\dc \doteq \left \{(\dot{\vt{x}}_s,\vt{x}_s,u_s)\right \}$ is corrupted by a sample (offline) noise bounded by $\epsilon$. The consistency set $\mathcal{C}$, 
which contains all systems that are consistent with the data, is defined as:
\begin{equation}
    \mathcal{C} \doteq \left\{
    f,g\colon \|\dot{\vt{x}}_s-f(\vt{x}_s)-g(\vt{x}_s)u_s\|_\infty \leq \epsilon, s=t_1..t_T \right\}.
\end{equation}
Recall that $f = \mt{F} \vt{\phi}$, $g = \mt{G} \vt{\gamma}$. Exploiting the following property of the Kronecker product \cite{petersen2008matrix} $$\text{vec}(\mt{P}^T\mt{X}\mt{Q}^T) = (\mt{Q}\otimes \mt{P}^T) \text{vec}(\mt{X}),$$ 
leads to the equivalent representation
\begin{equation} \label{eq:p1}
    \mathcal{C} = \left\{ 
    \vt{f},\vt{g} \colon 
    \begin{bmatrix}
        \mt{A} & \mt{B} \\ -\mt{A} & -\mt{B}
    \end{bmatrix} 
    \begin{bmatrix}
        \vt{f} \\ \vt{g}
    \end{bmatrix}
    \leq 
    \begin{bmatrix}
        \epsilon \mathbf{1} + \vt{\xi} \\ \epsilon \mathbf{1} - \vt{\xi}
    \end{bmatrix} \right\},
\end{equation}
where $\vt{f} = \text{vec} (\mt{F}^T)$, $\vt{g} = \text{vec} (\mt{G}^T)$ and
\begin{equation} \label{eq:abxi}
    \mt{A} \doteq \begin{bmatrix}
    \mt{I} \otimes \vt{\phi}^T(t_1) \\ \vdots \\ \mt{I} \otimes \vt{\phi}^T(t_T)
    \end{bmatrix},
    \mt{B} \doteq \begin{bmatrix}
    \mt{I} \otimes u\vt{\gamma}^T(t_1) \\ \vdots \\ \mt{I} \otimes u\vt{\gamma}^T(t_T)
    \end{bmatrix},
    \vt{\xi} \doteq 
    \begin{bmatrix}
    \dot{\vt{x}}(t_1) \\ \vdots \\ \dot{\vt{x}}(t_T)
    \end{bmatrix}.
\end{equation}

%\begin{assumption}
%The consistency set $\mathcal{C}$ is compact. %\label{assum:compact}
% ,thus enough data need to be collected until the matrix $ \begin{bmatrix} A & B \\ -A & -B \end{bmatrix}$ has full column rank. \urg{Reason?}
%\end{assumption}
%\begin{remark} A necessary condition for Assumption \ref{assum:compact} to hold is that $[A, B; -A, -B]$ has full column rank.
%\end{remark}
In order to establish robust safety, we need to add to this representation a description of all admissible disturbances. In the sequel, we will assume that this set  has a polytopic description of the form $\mathcal{W} \doteq \left \{ \vt{w} \colon \mt{W} \vt{w} \leq \vt{d}_\vt{w} \right \}$. Combining this description with the description of $\mathcal{C}$ leads to an augmented consistency set describing the set of all possible plants and disturbances:
\begin{equation}\label{eq:P1}
    \mathcal{P}_1 \doteq \left\{ \vt{f,g,w} \colon
    \begin{bmatrix}
        \mt{A} & \mt{B} & \mathbf{0}\\
        -\mt{A} & -\mt{B} & \mathbf{0}\\
        \mathbf{0} & \mathbf{0} & \mt{W}
    \end{bmatrix}
    \begin{bmatrix}
        \vt{f}\\ \vt{g} \\ \vt{w}
    \end{bmatrix}
    \leq 
    \begin{bmatrix}
        \epsilon \mathbf{1} + \vt{\xi}\\
        \epsilon \mathbf{1} - \vt{\xi}\\
        \vt{d}_\vt{w}\\
        \end{bmatrix} \right\}.
\end{equation}

%\begin{remark}
%The polytope $\mathcal{P}_1$ is defined by $2 nT + n_w$ linear %inequality constraints. Redundant constraints may be eliminated %using iterative linear programming \cite{caron1989degenerate}.
%\end{remark}

It follows that a pair $(\rho,\psi)$ solves Problem \ref{pb:1} if 
\beq \label{eq:allfgw}
\nabla \cdot [\rho f (\vt{x}) + \psi g(\vt{x}) + \rho \vt{w}] - \rho(\vt{x})h(\vt{x})> 0 \eeq
holds for all $\vt{x}$ and all $(\vt{f},\vt{g},\vt{w}) \in \mathcal{P}_1$. In principle, this condition can be reduced to an \ac{SOS} optimization over the coefficients of $\rho,\psi$ by a straight application of Putinar's Positivstellensatz \cite{putinar1993compact}. However, this approach quickly becomes intractable. As we show next, computational complexity can be substantially reduced by exploiting duality. 

For a given pair $(\rho,\psi)$, consider the set of all systems of the form \eqref{eq:dynamics_noise} that are rendered safe by the control action $u=\frac{\psi}{\rho}$, along with the corresponding admissible perturbations, that is, the set of all
$(\vt{f},\vt{g},\vt{w})$ such that \eqref{eq:allfgw} holds for all $\vt{x} \in \R^n$. For each $\vt{x}$, this set is a polytope of the form:
\begin{equation}
    \mathcal{P}_2 \doteq \left\{\vt{f},\vt{g},\vt{w}\colon -
    \begin{bmatrix}
        (\nabla\cdot(\rho\vt{\phi}^T))^T \\  
        (\nabla\cdot(\psi\vt{\gamma}^T))^T\\
        (\nabla\rho)^T
    \end{bmatrix}^T
    \begin{bmatrix}
        \vt{f} \\ \vt{g} \\ \vt{w}
    \end{bmatrix} < -\rho h\right\}.
\end{equation}

It follows that \eqref{eq:allfgw} holds for all admissible disturbances $\vt{w} \in \mathcal{W}$ and all plants in the consistency $\mathcal{C}$ set if and only if $\mathcal{P}_1 \subseteq \mathcal{P}_2$.  This inclusion can be enforced through duality as follows:
%the extended Farkas' Lemma as follows:

\begin{lemma}\label{lem:farkas}   Assume that the data and priors are consistent (e.g. $\mathcal{C} \not = \emptyset$) and that  enough data has been collected so that $\mathcal{C}$ is compact. Then 
$\mathcal{P}_1 \subseteq \mathcal{P}_2$ if and only if there exists a vector function $\vt{y}(\vt{x})\geq 0, \vt{y}(\vt{x}) \in \R^{2nT+2n}$ such that the following functional set of affine constraints is feasible:
\begin{equation}\label{eq:EFL}
    \vt{y}^T(\vt{x}) \mt{N} = \vt{r}(\vt{x}) \; \text{and} \;
    \vt{y}^T(\vt{x}) \vt{e}  < -\rho(\vt{x})h(\vt{x})
\end{equation}
where
\begin{equation}
    \begin{aligned} \label{eq:Ndef}  & \mt{N} \doteq 
    \begin{bmatrix}
    \mt{A} & \mt{B} & \mathbf{0}\\
    -\mt{A} & -\mt{B} & \mathbf{0} \\
    \mathbf{0} & \mathbf{0} & \mt{W}\\
    \end{bmatrix}, 
    \; \vt{e} \doteq 
    \begin{bmatrix}
        \epsilon \mathbf{1} + \vt{\xi}\\
        \epsilon \mathbf{1} - \vt{\xi}\\
      \vt{d}_\vt{w}
    \end{bmatrix}, \\
    & \vt{r}(\vt{x})\doteq -
    \begin{bmatrix}
        \nabla\cdot(\rho\vt{\phi}^T) & \nabla\cdot(\psi\vt{\gamma}^T) & \nabla\rho
    \end{bmatrix}.
    \end{aligned}
\end{equation}
\end{lemma}
\begin{proof} Since $\mathcal{C}$ is compact, from section 5.8.3 in \cite{boyd2004convex} it follows that the systems of inequalities
\beq \label{eq:alternatives}
  %  \begin{aligned}
       \begin{bmatrix} \mt{N} \\ -\vt{r} \end{bmatrix}  \begin{bmatrix}
        \vt{f} \\ \vt{g} \\ \vt{w}
    \end{bmatrix} \leq \begin{bmatrix} \vt{e} \\ \rho h \end{bmatrix} \; \textrm{and} \;\begin{array}{r}
    \vt{y}^T\mt{N}-\mu \vt{r} = 0 \\
     \vt{y}^T\vt{e} + \mu \rho h < 0 \\
    \vt{y} \geq 0, \; \mu \geq 0
    \end{array}
   % \end{aligned}
\eeq
are strong alternatives. Further, since $\mathcal{C} \not = \emptyset$ and $\mu >0$, we can take $\mu=1$ without loss of generality. Thus \eqref{eq:EFL} holds if and only if the left set of inequalities in \eqref{eq:alternatives} is infeasible. This implies that if \eqref{eq:EFL} holds,  a triple $(\vt{f,g,w}) \in \mathcal{P}_1$ if and only if 
$\begin{bmatrix}\vt{f}^T\; \vt{g}^T\;\vt{w}^T \end{bmatrix} \vt{r}^T < -\rho h $, that is  $(\vt{f,g,w}) \in \mathcal{P}_2$.
%Follows from direct application of the Extended Farkas' Lemma.
\end{proof}
\begin{remark} Proceeding as in Theorem 2 in \cite{dai2020semi}, it can be shown that  if $\vt{\phi}(\vt{x}), \vt{\gamma}(\vt{x})$ are continuous functions, then $\vt{y}(\vt{x})$ can be chosen to be continuous.
%\textcolor{red}{Further, if  $\vt{\phi},\vt{\gamma},\rho,\psi$ are polynomial and the trajectories of the closed-loop system stay in a compact region, then, without loss of generality $\vt{y}({x})$  can be taken to be polynomial. 
\end{remark}

Combining the observations above leads to
the main result of this paper:
\begin{theorem}\label{thm:main}
% Given noisy data of a , a sufficient condition for there to  exist 
A sufficient condition for the existence of a state-feedback control law $u(\vt{x})$ such that all systems in the consistency set $\mathcal{C}$ are rendered robustly safe, is that there exists a continuous vector function $\vt{y}(\vt{x}) \geq 0$ and functions $\rho \in C^1$, $\psi\in C^1$ 
%with $\frac{\psi(x)}{\rho(x)} \in C^0$
such that
\begin{subequations}\label{eq:thm_main} 
\begin{align}
\vt{y}^T(\vt{x})\mt{N} &= \vt{r}(\vt{x}), \ \forall \vt{x} \in \R^n\label{eq:thm_main1}\\
\vt{y}^T(\vt{x})\vt{e}  & < -\rho(\vt{x})h(\vt{x}), \ \forall \vt{x} \in \R^n \label{eq:thm_main2}\\
|\psi(\vt{x})| &\leq -\rho(\vt{x})h(\vt{x}),\ \forall \vt{x}\in \R^n \label{eq:thm_main3}\\
\rho(\vt{x}) & \geq 0,\ \forall \vt{x}\in\mathcal{X}_0\\ 
\rho(\vt{x}) & < 0, \ \forall \vt{x}\in\mathcal{X}_u.
\end{align}
\end{subequations}
The control law $u$ can then be extracted by the division $u(\vt{x}) = \psi(\vt{x})/\rho(\vt{x})$.
% And the control law can be extracted as $u=\psi/\rho$, which is a rational polynomial function. 
\end{theorem}
\begin{proof} The proof follows from the fact that from Lemma \ref{lem:farkas}, \eqref{eq:thm_main1} and \eqref{eq:thm_main2} guarantee that \eqref{eq:allfgw} holds for all plants in $\mathcal{C}$ and all admissible disturbances $\vt{w} \in \mathcal{W}$. Hence the conditions in Lemma \ref{lem:lemma2} hold for all plants that could have generated the observed data.
\end{proof}
    
\begin{remark}
Note that \eqref{eq:thm_main3} is a convex tightening of the condition that $\psi=0$ when $\rho=0$ %Simultaneously, it  leads to a bounded-control $|u|\leq c$ as $u=\psi/\rho$ by construction.
in the safe region  $\rho(\vt{x}) \geq 0 $.
%and for the safety concern, we %focus on the region where it %is safe.
\end{remark}

\subsection{Sum-of-Squares Safety Program}
\label{sec:sos_safety}
In order to solve the infinite-dimensional Problem \eqref{eq:thm_main} in a tractable manner, we restrict the variables $\rho, \psi, \vt{y}$ to be polynomials. Under this polynomial restriction, the extracted controller $u(\vt{x}) = \psi(\vt{x})/\rho(\vt{x})$ is then a rational function.
% Note that \eqref{eq:thm_main} is a feasibility polynomial problem, which can be solved using SOS program. 

% . The set of \ac{SOS} polynomials is a strict subset of all nonnegative polynomials, but 

% \subsection{SOS Polynomials}
% This paper shows that the data-driven safe control synthesis can be formulated as a polynomial nonnegativity  problem, which is generally NP-hard (for polynomials of even degree $d\geq 4$). This difficulty can be handled by relaxing the problem to establishing that a given polynomial is a Sum of Squares. Consider a polynomial $f$ of degree $2m$. A sufficient condition to $\forall x \in \R^n: f(x)\geq 0$ is that $f$ can be written as a \ac{SOS} 
% $f = \textstyle \sum_i g_i^2.$
% Let $v$ be a vector with all monomials of degree less than or equal to $m$. Then $f$ is \ac{SOS} if and only if there exists a positive semidefinite matrix $Q$ such that
% $f = v^T Qv$ \cite{parrilo2000structured}.
% Comparing term on the both sides leads to an \ac{SDP} program, which can be solved by off-the-shelf solvers. \textcolor{red}{do we need to mention the Psatz?}

Let $\mathcal{X}_0 \doteq \{\vt{x}\colon k(\vt{x}) \geq 0\}$ and $\mathcal{X}_u \doteq \{\vt{x}\colon h(\vt{x}) \geq 0\}$ denote the initial condition and unsafe sets, respectively. Algorithm 1 is \iac{SOS}-based finite-degree tightening of \eqref{eq:thm_main} for robustly safe control. Successful execution of algorithm \ref{alg:1} is sufficient for finding a robustly safe control law. 
%\urg{note that (A.2) imposes non-negativity instead of positivity}
% a detailed algorithm for data-driven robustly safe control is given in Alg. \ref{alg:1}.
\begin{algorithm}[h]
    \caption{Data-Driven Safe Control Design} \label{alg:1}
    \begin{algorithmic}
        \State Input: sample data $\dc$, and degrees $d_f,d_g,d_\rho,d_\psi$
        \State Let $2d_1\geq\text{max}\left\{d_f+d_\rho,d_g+d_\psi \right\}, 2d_2 \geq \text{max}\left\{d_\rho,d_\psi \right\}$
        \State Solve: the feasibility problem
        \[ \begin{array}{rcl}
        \textrm{coeff}_\vt{x}(\vt{y}^T \mt{N} - \vt{r})&=0                              \hfill(A.1) \\
        -\rho h - \vt{y}^T \vt{e} - c_1\; &\in \Sigma_{d_1}[\vt{x}] \hspace{1cm}   \hfill(A.2)\\
        -\rho h - \psi   \; &\in \Sigma_{d_2}[\vt{x}]   \hfill(A.3)\\
        -\rho h + \psi   \; &\in \Sigma_{d_2}[\vt{x}]   \hfill(A.4)\\
        \rho - s_1k  \; &\in \Sigma_{d_2} [\vt{x}]  \hfill(A.5)\\
        -\rho - s_2h - c_2\; &\in \Sigma_{d_2} [\vt{x}]  \hfill(A.6)\\
        \vt{y}_i            \; &\in \Sigma_{d_1}[\vt{x}]    \hfill(A.7) \\
        s_1, s_2       \; &\in \Sigma_{d_2}[\vt{x}]    \hfill(A.8) \\
        c_1, c_2       \; &> 0    \hfill(A.9)\\
        \end{array} \]
        \State Output: the safe control law $u = \psi/\rho$ or a certificate of infeasibility at degree $(d_1, d_2)$ 
    \end{algorithmic}
\end{algorithm}

% \urg{note about lack of convergence guarantee?}

\subsection{Computational Complexity Analysis}
\label{sec:complexity}

A straightforward application of Putinar's Positivstellensatz to solve \eqref{eq:allfgw} requires considering polynomials in the indeterminates $(\vt{x},\vt{f},\vt{g},\vt{w})$ with a total dimension $d_p=d_f+d_g+2n$. Thus, for an \ac{SOS} relaxation of order $d_r$, the total number of variables (hence the maximal size of Gram matrices) in the optimization is $\binom{d_r+d_p}{d_r}$. In contrast, by exploiting duality, Algorithm \ref{alg:1} only requires Gram matrices of maximal size $\binom{2+d_r}{d_r}$.

As an example, for a second order system with polynomial dynamics of degree 2, we have $d_f=d_g=6$. If $\rho$ and $\psi$ are also limited to degree 2 polynomials, for a relaxation of order $d_r=3$, the maximal Gram matrix size drops from $\binom{19}{3} = 969$ to $\binom{5}{3} = 10$.
\section{Numerical Examples} \label{sec:examples}

The proposed algorithm is tested on a pair of examples. Both experiments are implemented in MATLAB 2020b with Yalmip \cite{lofberg2004yalmip} and solved by Mosek \cite{mosek92}. Code to generate experiments and plots is publicly available at \url{https://github.com/J-mzz/ddc-safety}. %All experiments have an input bound of $c=1$.

\begin{example} Consider the Flow system \cite{rantzer2004analysis} with
\begin{align}
\label{eq:expr_1_dyn}
    f &= \begin{bmatrix}
    x_2 \\ -x_1 + \frac{1}{3} x_1^3 - x_2
    \end{bmatrix}, &
    g &= \begin{bmatrix}
    0 \\ 1
    \end{bmatrix}.
\end{align}

The initial and unsafe sets are the (unions of) disks:
% $$
% \begin{aligned}
%     \mathcal{X}_0 & = \{x \mid 0.25-x_1^2-(x_2+3)^2 \geq 0\},\\
%     \mathcal{X}_{u_1} & = \{x \mid 0.16-(x_1+1)^2-(x_2+1)^2\geq 0\},\\
%     \mathcal{X}_{u_2} & = \{x \mid 0.16-(x_1+1)^2-(x_2-1)^2\geq 0\}.
% \end{aligned}
% $$
$$
\begin{aligned}
    \mathcal{X}_0 = \{\vt{x} \mid \  &0.25-x_1^2-(x_2+3)^2 \geq 0\},\\
    \mathcal{X}_u = \{\vt{x} \mid \ &h_1(\vt{x})=0.16-(x_1+1)^2-(x_2+1)^2\geq 0 \ , \\
    \textrm{OR} \ &h_2(\vt{x})=0.16-(x_1+1)^2-(x_2-1)^2\geq 0\}.
\end{aligned}
$$
% \begin{figure}[b]
%     \centering
%     \includegraphics[width=0.7\linewidth]{figures/noisy data.png}
%     \caption{Noisy and observed data for system \eqref{eq:expr_1_dyn} \urg{not a good example: noisy and GT are on top of each other. Reviewers will kill us}}
%     \label{fig:noisy_data} 
% \end{figure}

% \urg{

% Report the number of faces and vertices in the polytope $\mathcal{P}_1$.

% What is the rational control law that you discover? Write $\rho, \psi_1, \psi_2$

% Consider plotting the reference trajectory, or the clean/corrupted data that you used to learn. Look at \href{https://github.com/Jarmill/data_driven_occ/blob/27d9a2e28b7882f6dbf795c3792a2cde2a7f208b/data_class/data_generator/data_generator.m#L296}{data\_plot\_2} for a method to plot the data vectors.

% }
\begin{figure}
     \centering
          \begin{subfigure}[b]{0.49\columnwidth}
         \centering
         \includegraphics[width=\linewidth]{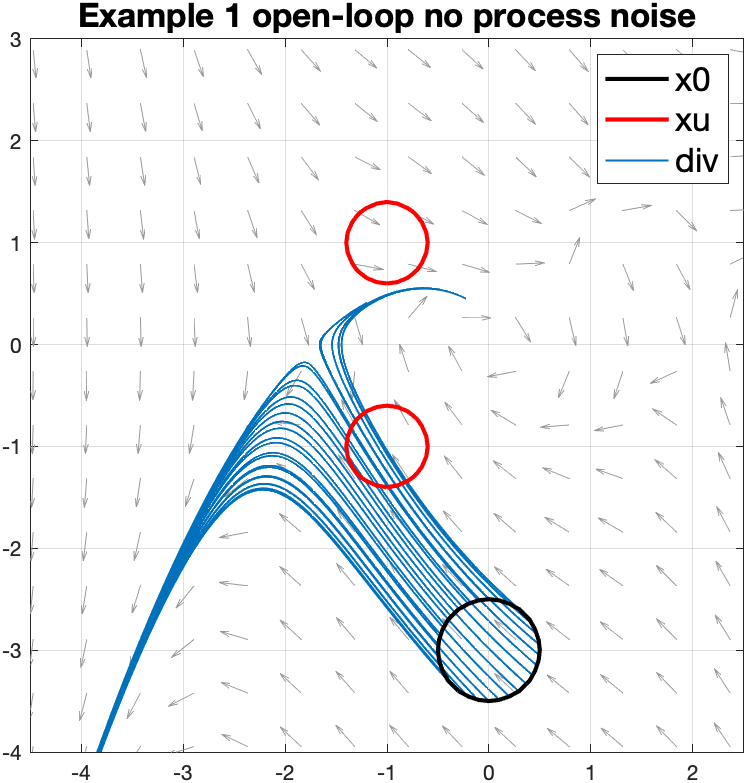}
         \caption{open-loop}
         \label{fig:ex1_open}
     \end{subfigure}
     \hfill
     \begin{subfigure}[b]{0.49\columnwidth}
         \centering
         \includegraphics[width=\linewidth]{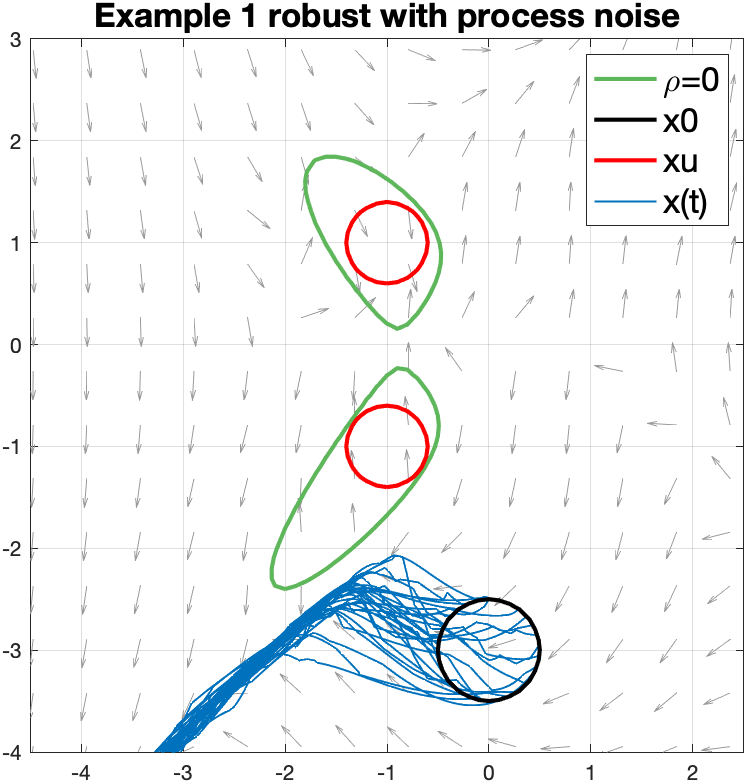}
         \caption{robust closed-loop}
         \label{fig:ex1_robust}
     \end{subfigure}
        \caption{Flow \eqref{eq:expr_1_dyn} simulations for Example 1}
        \label{fig:sim_ex1}
        \vspace{-1.5em}
\end{figure}

\begin{figure}
     \centering
     \begin{subfigure}[b]{0.49\columnwidth}
         \centering
         \includegraphics[width=\linewidth]{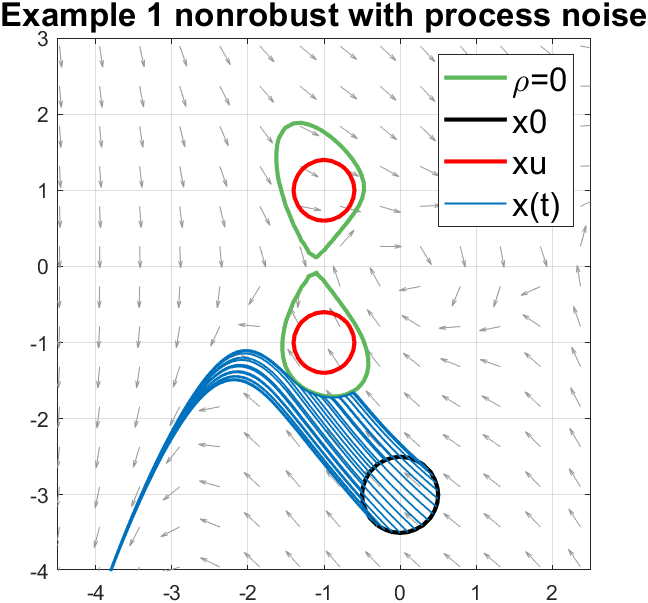}
         \caption{no process noise}
         \label{fig:ex1_nonrobust_no_noise}
     \end{subfigure}
     \hfill
     \begin{subfigure}[b]{0.49\columnwidth}
         \centering
         \includegraphics[width=\linewidth]{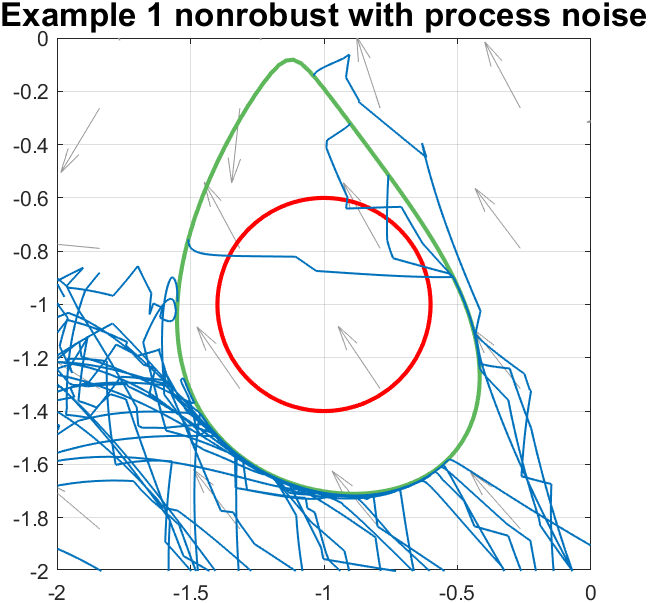}
         \caption{with process noise}
         \label{fig:ex1_nonrobust}
     \end{subfigure}
        \caption{Safe controllers synthesized without process noise may be unsafe when process noise is applied}
        \label{fig:sim_ex1_2}
        \vspace{-1.5em}
\end{figure}

% \begin{figure}[b]
%     \centering
%     \includegraphics[width=0.7\linewidth]{figures/nonrobust.png}
%     \caption{Example 1 designed without online noise}
%     \label{fig:ex1 w/o online noise} 
% \end{figure}

Results of the control design for Example 1 are shown in Fig. \ref{fig:sim_ex1} and \ref{fig:sim_ex1_2}. In each figure, 30 trajectories (blue curves) start from within the initial set $\mathcal{X}_0$ (black circle). The unsafe set $\mathcal{X}_u$ is the pair of red disks, implemented as $h(\vt{x}) = -h_1(\vt{x})h_2(\vt{x}) \geq 0$. Some of the open-loop trajectories in Fig. \ref{fig:ex1_open} enter the unsafe set $\mathcal{X}_u$ when starting in $\mathcal{X}_0$.

The prior knowledge of the system model is that $f$ is a two-dimensional cubic polynomial vector with $f(\mathbf{0})=\mathbf{0}$ and that $g$ is a two-dimensional constant vector,
%$f = [ \textrm{cubic}_1(x_1, x_2); \textrm{cubic}_2(x_1, x_2)]$ with $f([0; 0]) = [0; 0]$ and $g = [\textrm{constant}_1; \textrm{constant}_2]$, 
where the cubic polynomials in $f$ and the constant terms in $g$ are both unknown. 80 datapoints were collected and used to design a robustly safe controller under a sampling noise and a process noise bound of $\epsilon=\epsilon_\vt{w}=2$, yielding a polytope $\mathcal{P}_2$ from \eqref{eq:P1} with 22 dimensions ($\text{dim}(\vt{f})=18, \ \text{dim}(\vt{g})=2, \ \text{dim}(\vt{w})=2$) and 324 faces (91 of the faces $\mathcal{P}_2$ are nonredundant \cite{caron1989degenerate}). Algorithm \ref{alg:1} was used to find $\rho,\psi \in \R[\vt{x}]_{\leq 4}$, yielding 99 Gram matrices of maximal size $\binom{6}{4} = 15$ and the rational control law $u=\psi/\rho$. Fig. \ref{fig:ex1_robust} plots trajectories associated with this safe control law, and also features the $\rho=0$ level set in green.

Fig. \ref{fig:sim_ex1_2} highlights the importance of robustness in execution as well as in data-collection. The controller in Fig. \ref{fig:sim_ex1_2} was computed with the same noisy observed data as in Fig. \ref{fig:sim_ex1} but with $\epsilon_\vt{w} = 0$. The left plot in Fig. \ref{fig:ex1_nonrobust_no_noise} shows that the control is safe under noiseless trajectory execution. The right plot is zoomed into the lower red disk, and  demonstrates that some controlled trajectories pass through the $\rho=0$ contour and  enter $\mathcal{X}_u$ when process noise $\norm{\vt{w}}_\infty\leq 2$ is applied in execution (trajectories are terminated when $u\geq 10^4$, which is caused by numerical issues near the $\rho=0$ contour).

To summarize this example, $\rho\geq 0$ is an invariant set for all consistent systems under online noise when the robust controller is applied. The level set $\rho=0$ will separates initial set $\mathcal{X}_0$ and the unsafe set $\mathcal{X}_u$. Uncontrolled (Fig. \ref{fig:ex1_open}) and  non-robustly-safe (Fig. \ref{fig:ex1_nonrobust}) trajectories may enter $\mathcal{X}_u$.
\end{example}

\begin{example} Consider the Twist system \cite{miller2022bounding} with:

\begin{align}
\label{eq:expr_2_dyn}
    f&= \begin{bmatrix}-2.5x_1 + x_2 - 0.5x_3 + 2x_1^3+2x_3^3 \\
    -x_1+1.5x_2+0.5x_3-2x_2^3-2x_3^3 \\
    1.5 x_1 + 2.5x_2 - 2 x_3 - 2x_1^3 - 2 x_2^3\end{bmatrix}, & g&= \begin{bmatrix} 0 \\ 0\\  1 \end{bmatrix}.
\end{align}

% \begin{align}
% \label{eq:expr_2_dyn}
%     f = \sum_j \mt{E_{ij}}x_j -\mt{F_{ij}}(4x_j^3-3x_j)/2, \ \
%     g =  [0, 0, 1]^T,\\
%     \mt{E} = \begin{bmatrix}
%     -1&1&1 \\ -1&0&-1\\ 0&1&-2
%     \end{bmatrix}, \ \ 
%     \mt{F} = \begin{bmatrix}
%     -1&0&-1 \\ 0&1&1\\ 1&1&0
%     \end{bmatrix}.
% \end{align}
The initial and unsafe sets are the spheres:
$$
\begin{aligned}
    \mathcal{X}_0 & = \{\vt{x} \mid 0.01-(x_1+0.5)^2-x_2^2-x_3^2 \geq 0\},\\
    \mathcal{X}_u & = \{\vt{x} \mid 0.01-(x_1+0.1)^2-x_2^2-x_3^2\geq 0\}.
\end{aligned}
$$

Results of Example 2 are are shown in Fig. \ref{fig:sim_ex2}. Trajectories start within the initial set $\mathcal{X}_0$ (black sphere), and some of the open-loop trajectories in Fig. \ref{fig:ex2_wo_u} will enter the unsafe set $\mathcal{X}_u$ (red sphere). The prior knowledge of the system model is that $f$ is a three-dimensional cubic polynomial vector with $f(\mathbf{0})=\mathbf{0}$ and that $g$ is a three-dimensional constant vector. 80 datapoints were collected and used to design a robust safe controller under a sampling noise and a process noise bound of $\epsilon=\epsilon_\vt{w}=1$, yielding a polytope $\mathcal{P}_2$ with 63 dimensions ($\text{dim}(\vt{f})=38, \ \text{dim}(\vt{g})=3, \ \text{dim}(\vt{w})=3$) and 304 faces (all of them are nonredundant). Using Algorithm \ref{alg:1} to find $\rho,\psi \in \R[\vt{x}]_{\leq 4}$ yields a rational control law $u=\psi/\rho$. Fig. \ref{fig:ex2_w_u} features the $\rho=0$ level set surface in green.

\begin{figure}
     \centering
     \begin{subfigure}[b]{0.49\columnwidth}
         \centering
         \includegraphics[width=\linewidth]{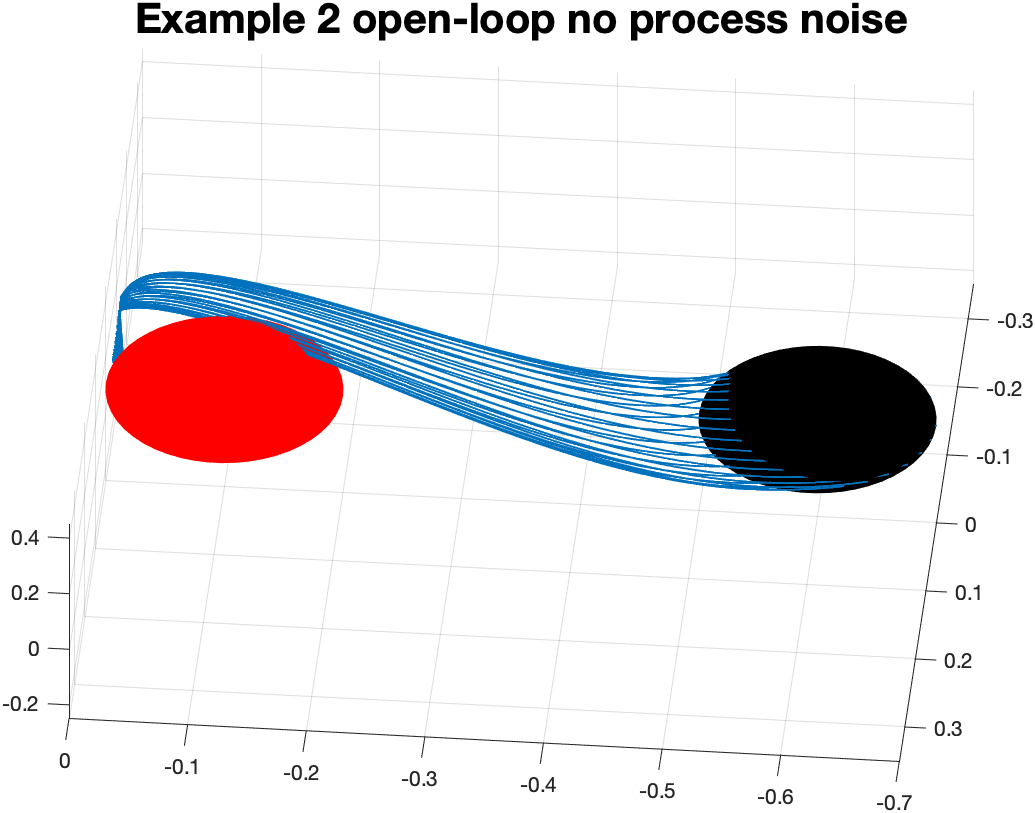}
         \caption{open-loop}
         \label{fig:ex2_wo_u}
     \end{subfigure}
     \hfill
     \begin{subfigure}[b]{0.49\columnwidth}
         \centering
         \includegraphics[width=\linewidth]{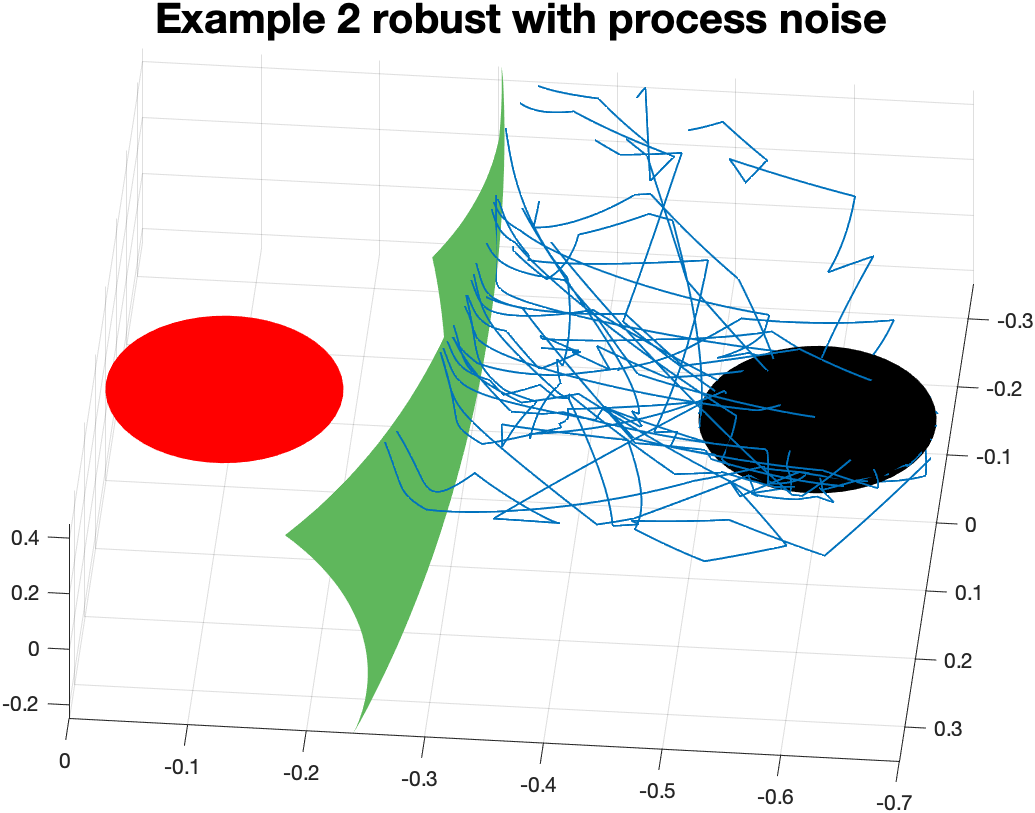}
         \caption{robust closed-loop}
         \label{fig:ex2_w_u}
     \end{subfigure}
        \caption{Twist \eqref{eq:expr_2_dyn} simulations for Example 2.}
        \label{fig:sim_ex2}
        \vspace{-1.5em}
\end{figure}
\end{example}

% \old{Results of the control design are shown in Fig. \ref{fig:sim_ex2}. The black and the red spheres represent the initial and the unsafe set respectively, the green surface is the level set of $\rho =0$.
% Trajectories (open-loop in Fig. \eqref{fig:ex2_wo_u} and closed-loop in Fig. \eqref{fig:ex2_w_u}) are displayed in the 30 blue lines starting on the boundary of $\mathcal{X}_0$.}
\section{Conclusion} \label{sec:conclusion}

This paper uses density functions to find provably safe controllers for systems whose data-observations and executions are both corrupted by $L_\infty$-bounded noise. The output of Algorithm \ref{alg:1} (if successful) is a rational controller $u$, along with a density certificate $\rho$ that guarantees robust safety of all trajectories starting in the initial set. Future work involves steering safe trajectories to a destination set (e.g. reach avoid, asymptotic stability), adding performance objectives, and extension to other noise and disturbance models (e.g. $L_2$ or semidefinite bounded signals).

% \input{sections/appendix}

% %%%%%%%%%%%%%%%%%%%%%%%%%%%%%%%%%%%%%%%%%%%%%%%%%%%%%%%%%%%%%%%%%%%%%%%%%%%%%%%%
% \section{Acknowledgements}

% The authors thank Milan Korda for his discussions about occupation measures and time-varying uncertainty.

\bibliographystyle{IEEEtran}
\bibliography{references.bib}

\begin{thebibliography}{10}
\providecommand{\url}[1]{#1}
\csname url@samestyle\endcsname
\providecommand{\newblock}{\relax}
\providecommand{\bibinfo}[2]{#2}
\providecommand{\BIBentrySTDinterwordspacing}{\spaceskip=0pt\relax}
\providecommand{\BIBentryALTinterwordstretchfactor}{4}
\providecommand{\BIBentryALTinterwordspacing}{\spaceskip=\fontdimen2\font plus
\BIBentryALTinterwordstretchfactor\fontdimen3\font minus
  \fontdimen4\font\relax}
\providecommand{\BIBforeignlanguage}[2]{{%
\expandafter\ifx\csname l@#1\endcsname\relax
\typeout{** WARNING: IEEEtran.bst: No hyphenation pattern has been}%
\typeout{** loaded for the language `#1'. Using the pattern for}%
\typeout{** the default language instead.}%
\else
\language=\csname l@#1\endcsname
\fi
#2}}
\providecommand{\BIBdecl}{\relax}
\BIBdecl

\bibitem{prajna2004safety}
S.~Prajna and A.~Jadbabaie, ``{Safety Verification of Hybrid Systems Using
  Barrier Certificates},'' in \emph{HSCC}, vol. 2993.\hskip 1em plus 0.5em
  minus 0.4em\relax Springer, 2004, pp. 477--492.

\bibitem{ames2019control}
A.~D. Ames, S.~Coogan, M.~Egerstedt, G.~Notomista, K.~Sreenath, and P.~Tabuada,
  ``{Control Barrier Functions: Theory and Applications},'' in \emph{2019 18th
  European control conference (ECC)}.\hskip 1em plus 0.5em minus 0.4em\relax
  IEEE, 2019, pp. 3420--3431.

\bibitem{xiao2019control}
W.~Xiao and C.~Belta, ``{Control Barrier Functions for Systems with High
  Relative Degree},'' in \emph{2019 IEEE 58th conference on decision and
  control (CDC)}.\hskip 1em plus 0.5em minus 0.4em\relax IEEE, 2019, pp.
  474--479.

\bibitem{ames2014control}
A.~D. Ames, J.~W. Grizzle, and P.~Tabuada, ``{Control Barrier Function Based
  Quadratic Programs with Application to Adaptive Cruise Control},'' in
  \emph{53rd IEEE Conference on Decision and Control}.\hskip 1em plus 0.5em
  minus 0.4em\relax IEEE, 2014, pp. 6271--6278.

\bibitem{XU2015robustcbf}
X.~Xu, P.~Tabuada, J.~W. Grizzle, and A.~D. Ames, ``{Robustness of Control
  Barrier Functions for Safety Critical Control},'' \emph{IFAC-PapersOnLine},
  vol.~48, no.~27, pp. 54--61, 2015, analysis and Design of Hybrid Systems
  ADHS.

\bibitem{majumdar2013control}
A.~Majumdar, A.~A. Ahmadi, and R.~Tedrake, ``{Control Design along Trajectories
  with Sums of Squares Programming},'' in \emph{2013 IEEE International
  Conference on Robotics and Automation}.\hskip 1em plus 0.5em minus
  0.4em\relax IEEE, 2013, pp. 4054--4061.

\bibitem{rantzer2004analysis}
A.~Rantzer and S.~Prajna, ``{On Analysis and Synthesis of Safe Control Laws},''
  in \emph{42nd Allerton Conference on Communication, Control, and
  Computing}.\hskip 1em plus 0.5em minus 0.4em\relax University of Illinois,
  2004, pp. 1468--1476.

\bibitem{rantzer2001dual}
A.~Rantzer, ``A dual to {L}yapunov's stability theorem,'' \emph{Systems \&
  Control Letters}, vol.~42, no.~3, pp. 161--168, 2001.

\bibitem{chen2020densityvalue}
Y.~Chen, M.~Ahmadi, and A.~D. Ames, ``{Optimal Safe Controller Synthesis: A
  Density Function Approach},'' in \emph{2020 American Control Conference
  (ACC)}, 2020, pp. 5407--5412.

\bibitem{coogan2020mixed}
S.~Coogan, ``{Mixed Monotonicity for Reachability and Safety in Dynamical
  Systems},'' in \emph{2020 59th IEEE Conference on Decision and Control
  (CDC)}.\hskip 1em plus 0.5em minus 0.4em\relax IEEE, 2020, pp. 5074--5085.

\bibitem{bansal2017hamilton}
S.~Bansal, M.~Chen, S.~Herbert, and C.~J. Tomlin, ``{Hamilton-Jacobi
  Reachability: A Brief Overview and Recent Advances},'' in \emph{2017 IEEE
  56th Annual Conference on Decision and Control (CDC)}.\hskip 1em plus 0.5em
  minus 0.4em\relax IEEE, 2017, pp. 2242--2253.

\bibitem{brunke2022safe}
L.~Brunke, M.~Greeff, A.~W. Hall, Z.~Yuan, S.~Zhou, J.~Panerati, and A.~P.
  Schoellig, ``{Safe Learning in Robotics: From Learning-Based Control to Safe
  Reinforcement Learning},'' \emph{Annual Review of Control, Robotics, and
  Autonomous Systems}, vol.~5, pp. 411--444, 2022.

\bibitem{folkestad2020data}
C.~Folkestad, Y.~Chen, A.~D. Ames, and J.~W. Burdick, ``{Data-Driven
  Safety-Critical Control: Synthesizing Control Barrier Functions with Koopman
  Operators},'' \emph{IEEE Control Systems Letters}, vol.~5, no.~6, pp.
  2012--2017, 2020.

\bibitem{formentin2014comparison}
S.~Formentin, K.~Van~Heusden, and A.~Karimi, ``{A comparison of model-based and
  data-driven controller tuning},'' \emph{International Journal of Adaptive
  Control and Signal Processing}, vol.~28, no.~10, pp. 882--897, 2014.

\bibitem{dai2018moments}
T.~Dai and M.~Sznaier, ``{A Moments Based Approach to Designing MIMO Data
  Driven Controllers for Switched Systems},'' in \emph{2018 IEEE Conference on
  Decision and Control (CDC)}.\hskip 1em plus 0.5em minus 0.4em\relax IEEE,
  2018, pp. 5652--5657.

\bibitem{dai2020semi}
{T. Dai and M. Sznaier}, ``{A Semi-Algebraic Optimization Approach to
  Data-Driven Control of Continuous-Time Nonlinear Systems},'' \emph{IEEE
  Control Systems Letters}, vol.~5, no.~2, pp. 487--492, 2020.

\bibitem{waarde2020noisy}
H.~J. van Waarde, M.~K. Camlibel, and M.~Mesbahi, ``{From Noisy Data to
  Feedback Controllers: Nonconservative Design via a {M}atrix {S}-{L}emma},''
  \emph{IEEE Trans. Automat. Contr.}, 2020.

\bibitem{martin2021data}
T.~Martin and F.~Allg{\"o}wer, ``{Data-driven system analysis of nonlinear
  systems using polynomial approximation},'' \emph{arXiv preprint
  arXiv:2108.11298}, 2021.

\bibitem{berberich2021robustmpc}
J.~Berberich, J.~Köhler, M.~A. Müller, and F.~Allgöwer, ``{Data-Driven Model
  Predictive Control With Stability and Robustness Guarantees},'' \emph{IEEE
  Trans. Automat. Contr.}, vol.~66, no.~4, pp. 1702--1717, 2021.

\bibitem{bisoffi2022data}
A.~Bisoffi, C.~De~Persis, and P.~Tesi, ``{Data-driven control via Petersen’s
  lemma},'' \emph{Automatica}, vol. 145, p. 110537, 2022.

\bibitem{miller2022lpvqmi}
J.~Miller and M.~Sznaier, ``{Data-Driven Gain Scheduling Control of Linear
  Parameter-Varying Systems using Quadratic Matrix Inequalities},'' \emph{IEEE
  Control Systems Letters}, 2022.

\bibitem{miller2022eiv_short}
J.~Miller, T.~Dai, and M.~Sznaier, ``{Data-Driven Superstabilizing Control of
  Error-in-Variables Discrete-Time Linear Systems},'' in \emph{2022 61st IEEE
  Conference on Decision and Control (CDC)}, 2022, pp. 4924--4929.

\bibitem{rosolia2018learning}
U.~Rosolia and F.~Borrelli, ``{Learning Model Predictive Control for Iterative
  Tasks. A Data-Driven Control Framework},'' \emph{IEEE Transactions on
  Automatic Control}, vol.~63, no.~7, pp. 1883--1896, 2018.

\bibitem{lopez2021robust}
B.~T. Lopez, J.-J.~E. Slotine, and J.~P. How, ``{Robust Adaptive Control
  Barrier Functions: An Adaptive and Data-Driven Approach to Safety},''
  \emph{IEEE Control Systems Letters}, vol.~5, no.~3, pp. 1031--1036, 2021.

\bibitem{dacs2022robust}
E.~Da{\c{s}} and R.~M. Murray, ``{Robust Safe Control Synthesis with
  Disturbance Observer-Based Control Barrier Functions},'' in \emph{2022 IEEE
  61st Conference on Decision and Control (CDC)}.\hskip 1em plus 0.5em minus
  0.4em\relax IEEE, 2022, pp. 5566--5573.

\bibitem{parrilo2000structured}
P.~A. Parrilo, \emph{Structured Semidefinite Programs and Semialgebraic
  Geometry Methods in Robustness and Optimization}.\hskip 1em plus 0.5em minus
  0.4em\relax California Institute of Technology, 2000.

\bibitem{petersen2008matrix}
K.~B. Petersen, M.~S. Pedersen \emph{et~al.}, ``{The Matrix Cookbook},''
  \emph{Technical University of Denmark}, vol.~7, no.~15, p. 510, 2008.

\bibitem{putinar1993compact}
M.~Putinar, ``{Positive Polynomials on Compact Semi-algebraic Sets},''
  \emph{Indiana University Mathematics Journal}, vol.~42, no.~3, pp. 969--984,
  1993.

\bibitem{boyd2004convex}
S.~Boyd, S.~P. Boyd, and L.~Vandenberghe, \emph{{Convex Optimization}}.\hskip
  1em plus 0.5em minus 0.4em\relax Cambridge University Press, 2004.

\bibitem{lofberg2004yalmip}
J.~{Lofberg}, ``{YALMIP : A toolbox for modeling and optimization in MATLAB},''
  in \emph{ICRA (IEEE Cat. No.04CH37508)}, 2004, pp. 284--289.

\bibitem{mosek92}
\BIBentryALTinterwordspacing
M.~ApS, \emph{The MOSEK optimization toolbox for MATLAB manual. Version 9.2.},
  2020. [Online]. Available:
  \url{https://docs.mosek.com/9.2/toolbox/index.html}
\BIBentrySTDinterwordspacing

\bibitem{caron1989degenerate}
R.~Caron, J.~McDonald, and C.~Ponic, ``A degenerate extreme point strategy for
  the classification of linear constraints as redundant or necessary,''
  \emph{Journal of Optimization Theory and Applications}, vol.~62, no.~2, pp.
  225--237, 1989.

\bibitem{miller2022bounding}
J.~Miller and M.~Sznaier, ``{Bounding the Distance of Closest Approach to
  Unsafe Sets with Occupation Measures},'' in \emph{2022 IEEE 61st Conference
  on Decision and Control (CDC)}.\hskip 1em plus 0.5em minus 0.4em\relax IEEE,
  2022, pp. 5008--5013.

\end{thebibliography}

\end{document}